\documentclass[12pt]{amsart}

\usepackage{fullpage,url,amssymb,amsmath,amsthm,amsfonts,mathrsfs,latexsym}
\usepackage[usenames,dvipsnames]{color}
\usepackage[pagebackref = true, colorlinks = true, linkcolor = blue, citecolor = Green]{hyperref}
\usepackage[alphabetic,lite]{amsrefs}
\usepackage[all, cmtip]{xy}
\usepackage{xfrac}
\usepackage{appendix}

\newcommand{\edit}[1]{{#1}}
\newcommand{\Edit}[1]{}

\newcommand{\Z}{{\mathbb Z}}
\newcommand{\Q}{{\mathbb Q}}

\newcommand{\PP}{{\mathbb P}}
\newcommand{\G}{{\mathbb G}}

\newcommand{\Kbar}{{\overline{K}}}

\newcommand{\Lbar}{{\overline{L}}}

\newcommand{\Xbar}{{\overline{X}}}
\newcommand{\Ybar}{{\overline{Y}}}
\newcommand{\Vbar}{{\overline{V}}}

\newcommand{\calA}{{\mathcal A}}
\newcommand{\calB}{{\mathcal B}}
\newcommand{\calC}{{\mathcal C}}

\newcommand{\calI}{{\mathcal I}}
\newcommand{\calJ}{{\mathcal J}}

\newcommand{\calL}{{\mathcal L}}

\newcommand{\To}{\longrightarrow}

\DeclareMathOperator{\Map}{Map}

\DeclareMathOperator{\Hom}{Hom}

\DeclareMathOperator{\Norm}{Norm}

\DeclareMathOperator{\divv}{div}

\DeclareMathOperator{\Gal}{Gal}
\DeclareMathOperator{\Cor}{Cor}

\DeclareMathOperator{\Br}{Br}

\DeclareMathOperator{\Div}{Div}
\DeclareMathOperator{\Princ}{Princ}
\DeclareMathOperator{\Pic}{Pic}
\DeclareMathOperator{\Alb}{Alb}
\DeclareMathOperator{\Jac}{Jac}
\DeclareMathOperator{\HH}{H}
\DeclareMathOperator{\NS}{NS}

\DeclareMathOperator{\Spec}{Spec}
\DeclareMathOperator{\Aut}{Aut}
\DeclareMathOperator{\PGL}{PGL}

\newcommand{\kk}{\mathbf{k}}

\newcommand{\ind}{\edit{I}}
\newcommand{\per}{\edit{P}}

\newcommand{\diw}{\operatorname{div}}

\DeclareFontEncoding{OT2}{}{} 

\newcommand{\defi}[1]{\textsf{#1}} 

\newtheorem{Theorem}{Theorem}[section]
\newtheorem{Lemma}[Theorem]{Lemma}
\newtheorem{Proposition}[Theorem]{Proposition}
\newtheorem{Corollary}[Theorem]{Corollary}
\newtheorem{Definition}[Theorem]{Definition}
\newtheorem{Example}[Theorem]{Example}
\newtheorem{Remark}[Theorem]{Remark}

\numberwithin{equation}{section}

\begin{document}

\title{Relative Brauer groups of torsors of period two}

\author{Brendan Creutz}
\address{Department of Mathematics and Statistics, University of Canterbury, Private Bag 4800, Christchurch 8140, New Zealand}
\email{brendan.creutz@canterbury.ac.nz}
\urladdr{http://www.math.canterbury.ac.nz/\~{}bcreutz}

\maketitle

\begin{abstract}
We consider the problem of computing the relative Brauer group of a torsor of period two under an elliptic curve. We show how this problem can be reduced to finding a set of generators for the group of rational points on the elliptic curve. This extends work of Haile and Han to the case of torsors with unequal period and index. Our results also apply to torsors under higher dimensional abelian varieties. Several numerical examples are given.
\end{abstract}

\section{Introduction}

	For a geometrically integral variety $V$ over a field $K$, the relative Brauer group $\Br(V/K)$ is the subgroup of the Brauer group $\Br(K)$ consisting of classes which are split by the function field $\kk(V)$ of $V$. The index of $V$ is the greatest common divisor of the degrees of field extensions $K'/K$ for which $V$ possesses a $K'$-rational point. These two invariants are closely related. For example, the relative Brauer group can only be nontrivial when the index is greater than one. The simplest case in which this occurs is a conic without a rational point, for which the relative Brauer group is generated by the class of the conic in $\Br(K)$.

	For curves with nontrivial Jacobian, the relative Brauer group need not be cyclic. Here the simplest examples occur among genus one curves of index $2$, which (assuming the characteristic of $K$ is not $2$) are curves admitting a model of the form $y^2 = g(x)$ with $g(x) \in K[x]$ a square free quartic polynomial. For such a curve $V$ results of Shick~\cite{Shick}, Han~\cite{Han} and Haile and Han~\cite{HaileHan} reduce the problem of determining the relative Brauer group to determining the group of $K$-rational points on the Jacobian $J$ of $V$. Their approach can be explained by an old result of Lichtenbaum~\cite{Lichtenbaum69}: corresponding to any torsor $V$ there is a canonically defined class in $\calA_V \in \Br(J)$ such that the evaluation map induces a homomorphism $J(K) \to \Br(V/K)$ which, as noted by \c Ciperiani and Krashen~\cite{CipKrash}, will be surjective when the index of $V$ and the order of $\calA_V$ (called the period of $V$) coincide.
	
	In \cite{HaileHan} Lichtenbaum's result is made explicit by constructing a representative for $\calA_V$ as an Azumaya algebra over the coordinate ring of $J$. Their work was motivated by earlier work of Haile~\cite{Haile}, who showed that the Clifford algebra of a binary cubic form $F(x,y)$ plays an analogous role for the genus one curve $z^3 = F(x,y)$. This has been further developed by Kuo~\cite{Kuo} to construct the algebra $\calA_V$ corresponding to an arbitrary cubic curve, i.e., a torsor of period and index dividing $3$. A method applicable to torsors of cyclic type (necessarily having equal period and index) is given in \cite{CipKrash} and applied in~\cite{HaileHanWadsworth}.

	For a hyperelliptic curve $X$, Creutz and Viray~\cite{CVCurves} have recently given an explicit means of constructing unramified central simple $\kk(X)$-algebras representing $2$-torsion elements in the Brauer group of $X$. Our primary purpose here is to explore applications of these algebras to the computation of relative Brauer groups. We show that each of these algebras corresponds to a torsor $V$ under the Jacobian of $X$ in such a way that the specializations of the algebra give elements of $\Br(V/K)$. When $X$ is a genus one curve, the corresponding torsors are intersections of quadric hypersurfaces in $\PP^3$, and every torsor of index dividing $4$ is given this way, including those with unequal period and index. The results also apply to higher genus hyperelliptic curves allowing for computation of relative Brauer groups of higher dimensional torsors under abelian varieties. In Section~\ref{sec:Examples} we give several numerical examples illustrating the practicality of the method.

	\subsection{Statement of the results}

	Throughout the paper we consider a smooth and projective hyperelliptic curve $X$ defined over a field $K$ of characteristic different from $2$. We fix an equation $y^2 = cf(x)$ defining $X$, where $c \in K^\times$ and $f(x) \in K[x]$ is a monic square free polynomial. We set $L = K[\theta]/f(\theta)$ and let $J = \Jac(X)$ denote the Jacobian of $X$. For any $\ell \in L^\times$, let $\calA_\ell''$ be the $L(x)$-algebra obtained by adjoining anticommuting square roots of $x-\theta$ and $\ell$ to $L(x)$. The corestriction $\calA'_\ell := \Cor_{L(x)/K(x)}(\calA''_\ell)$ is a tensor product of quaternion algebras over $K(x)$ which can be explicitly computed using Rosset-Tate reciprocity (see~\cite[Proposition 3.1]{CVCurves}). In \cite{CVCurves} it is shown that, under suitable conditions on the norm of $\ell$, the algebra $\calA'_\ell$ pulls back to an unramified central simple $\kk(X)$-algebra $\calA_\ell := \calA'_\ell\otimes_{K(x)}\kk(X)$.

	\subsection{Elliptic curves}
		\label{subsec:EC}
	When $\deg(f) = 3$, $X = J$ is an elliptic curve and every torsor of period $2$ under $X$ is given by a classical construction associating to an element $\ell \in L^\times$ of square norm the genus one curve $V_\ell \subset \PP^3_K$ defined by
	\begin{equation}\label{eq:V_ell}
			V_\ell : Q_\ell^{(1)}(u,v,w) + t^2 = Q_\ell^{(2)}(u,v,w) = 0\,,		
	\end{equation}
	where, for $0 \le i \le 2$, $Q_\ell^{(i)}(u,v,w) \in K[u,v,w]$ is the ternary quadratic form uniquely determined by requiring that $\sum \theta^iQ_\ell^{(i)} = \ell(u + \theta v + \theta^2 w)^2$. See \cite[Section 15]{CasselsLectureNotes}. \edit{In the appendix we prove that every torsor of period $2$ arises in this way.}

	 The following theorem, proven in Section~\ref{subsec:EC2}, reduces computation of the relative Brauer group of $V_\ell$ to a computation of a basis for $X(K)/2X(K)$.

	\begin{Theorem}
		\label{thm:Elliptic}
		Let $V$ be a genus one curve of period $2$ and suppose $X:y^2 = f(x)$ is a Weierstrass model for $X = \Jac(V)$ (i.e., $\deg(f) = 3$). Then 
		\begin{enumerate}
			\item There exist $\ell \in L^\times$ such that $V \simeq V_\ell$.
			\item Evaluation of $\calA_\ell$ induces an exact sequence
				\[
					\frac{X(K)}{2X(K)} \stackrel{\calA_\ell}\To \Br(V/K) \To \frac{\per(V)}{\ind(V)} \To 0\,,
				\]
				where $\per(V)$ and $\ind(V)$ are the period and index subgroups in Definition~\ref{def:perind}.\Edit{As suggested by the referee the notation in Definition \ref{def:perind} has been changed.}
			\item $\Br(V/K)$ is generated by $\left\{ \calA_\ell(P) \;:\; P \in X(K) \right\}$ and the class $[C_\ell] \in \Br(K)$ corresponding to the conic $C_\ell \subset \PP^2_K$ defined by $Q^{(2)}_\ell(u,v,w)=0$.
			\item $\per(V) = \ind(V)$ if and only if $[C_\ell] \in \calA_\ell(X(K))$.
		\end{enumerate}
	\end{Theorem}

	The connection between these torsors and those considered in \cite{HaileHan} is as follows. Projecting away from the point $u=v=w=0$ in $\PP^3_K$ induces a degree $2$ morphism $V_\ell \to C_\ell$, where $C_\ell \subset \PP^2_K$ is the conic defined by $Q_\ell^{(2)}(u,v,w) = 0$. If $C_\ell$ has a $K$-rational point, then $C_\ell \simeq \PP^1_K$ and so $V_\ell$ is a double cover of $\PP^1_K$ with a model of the form $y^2 = $ quartic. However, as first shown by Cassels \cite{CasselsV}, it is possible that $V_\ell$ has no rational points over any quadratic extension of $K$, in which case $\per(V_\ell)\ne\ind(V_\ell)$ and no degree $2$ map to $\PP^1_K$ exists. In fact, the torsors $V_\ell$ with equal period and index are not generic; they correspond to those $\ell \in L^\times$ whose classes in $L^\times/L^{\times 2}$ can be represented by a linear polynomial in $K[\theta] = L$ (This follows easily from~\eqref{eq:V_ell}; see also~\cite[Theorem 13]{CremonaFisher}). Theorem~\ref{thm:Elliptic} therefore recovers results of \cite{HaileHan}, and extends them to all torsors of period $2$ under an elliptic curve.
	
	\subsection{Genus one curves}\label{subsec:G1}
When $X : y^2 = cf(x)$ is defined by a quartic polynomial $cf(x)$, a construction analogous to~\eqref{eq:V_ell} yields an intersection of quadric surfaces $V_\ell \subset \PP^3$. Given $\ell \in L^\times$ such that $N_{L/K}(\ell)$ is a square times the leading coefficient of $cf(x)$, one defines
	\begin{equation}\label{eq:V_ell2}
		V_\ell : Q_\ell^{(2)}(t,u,v,w) = Q_\ell^{(3)}(t,u,v,w) = 0,
	\end{equation}
	where, for $0\le i \le 3$, $Q_\ell^{(i)}(t,u,v,w)$ is the quadratic form defined by 
	\[
		\sum_{0\le i \le 3} \theta^iQ_\ell^{(i)}(t,u,v,w) = \ell(t + \theta u + \theta^2 v + \theta^3w)^2.
	\]
	We associate to $V_\ell$ the central simple $\kk(X)$-algebra $\calA_\ell$ constructed above. In Section \ref{subsec:C4} we prove the following.

	\begin{Theorem}\label{thm:index4}
		Let $V$ be a genus one curve of \edit{period and}\Edit{As noted by the referee, this assumption is used in the proof. And, indeed, the result as stated is not quite true without this assumption! However, the only case excluded by the new assumption is that when $V$ has period $2$, which is already covered in Theorem 1.1.} index $4$. Then there exists a quartic polynomial $cf(x) \in K[x]$ and $\ell \in L^\times$, where $L = K[\theta]/f(\theta)$ such that:
		\begin{enumerate}
			\item $\Jac(V) \simeq \Jac(X)$ where $X : y^2 = cf(x)$.
			\item $2V = X$ in the Weil-Ch\^atelet group of the Jacobian.
			\item $V \simeq V_\ell$.
			\item Evaluation of $\calA_\ell$ on degree zero $K$-rational divisors of $X$ gives rise to an exact sequence
				\[
					0 \to \Pic^0(V) \stackrel{\iota}\to \Pic^0(X) \stackrel{\calA_\ell}\To \Br(V/K) \stackrel{2}\To \Br(X/K) \To 0\,,
				\]
		where $\iota$ is induced by the inclusion $\Pic^0(V) \subset \Pic^0(\Vbar) = \Pic^0(\Xbar)$ and the last map is induced by multiplication by $2$ in $\Br(K)$.
		\end{enumerate}
	\end{Theorem}

	The first three statements of this theorem are well known. If $V$ has index dividing $4$, then the complete linear system associated to any $K$-rational divisor of degree $4$ gives a model for $V$ as an intersection of quadrics in $\PP^3_K$. The quartic polynomial is $\det\left( M_1x - M_2\right)$, where $M_1$ and $M_2$ are the symmetric matrices corresponding to these quadrics, and a formula for $\ell$ is given in~\cite[Section 5]{Fisher4d}.

	It is worth noting that while Theorem~\ref{thm:index4} may not yield enough to completely determine $\Br(V/K$), Theorems~\ref{thm:Elliptic} and~\ref{thm:index4} together allow one to determine whether a given $K$-rational point on the Jacobian is represented by a $K$-rational divisor of degree $0$ on $V$. See Section~\ref{subsec:period4} for an example.

	\subsection{Hyperelliptic curves}
	When the degree of $cf(x)$ is arbitrary, work of Cassels~\cite{CasselsGenus2}, Schaefer~\cite{Schaefer} and Poonen-Schaefer~\cite{PoonenSchaefer} recalled in Section~\ref{sec:HyperellipticCurves} similarly allows one to associate a torsor $V_\ell$ of period $2$ under $J$ to any $\ell \in L^\times$ of square norm. Consequently, to any such torsor we may associate an algebra $\calA_\ell$. The following more general version of~\ref{thm:Elliptic} is proven in Section~\ref{sec:HyperellipticCurves}.

	\begin{Theorem}
		\label{thm:Hyperelliptic}
		Suppose $V_\ell$ is the torsor of period $2$ under the Jacobian of a hyperelliptic curve $X$ corresponding to an element $\ell \in L^\times$ of square norm.  Suppose that $\Br(X/K) = 0$ and let $\calA_\ell$ be the corresponding $\kk(X)$-algebra. Then evaluation of $\calA_\ell$ on degree zero $K$-rational divisors of $X$ induces an exact sequence
		\[
			J(K)/2J(K) \stackrel{\calA_\ell}\To \Br(V_\ell/K) \To \frac{\per(V_\ell)}{\ind(V_\ell)} \To 0\,,
		\]
		where $\per(V_\ell)$ and $\ind(V_\ell)$ are the period and index subgroups defined in Definition~\ref{def:perind}.
	\end{Theorem}

	As mentioned above, Lichtenbaum's result shows that the relative Brauer group of a torsor with equal period and index under an elliptic curve $J$ can be obtained by speicialization of a Brauer class on $J$. More generally, \c Ciperiani and Krashen \cite[Theorem 3.5]{CipKrash} have shown this to be true for torsors under any abelian variety. In Section~\ref{sec:BrauerClassesViaSpecialization} we use their result to deduce the following more general version.

	\begin{Theorem}
		\label{thm:SpecializationOfBrauerClasses}
		Let $Y/K$ be a smooth projective and geometrically integral variety over $K$ with Picard variety $A = \Pic^0_Y$. There is a class $\calB_Y \in \Br(A)$ canonically associated to $Y$ (see Definition~\ref{def:B_X}) such that evaluation induces an exact sequence
		\[
			0 \To \frac{A(K)}{\Pic^0(Y)} \stackrel{\calB_Y}{\To} \Br(Y/K) \To \frac{\per(Y)}{\ind(Y)} \To 0,
		\]
	where $\per(Y)$ and $\ind(Y)$ are the period and index subgroups of Definition \ref{def:perind}.
	\end{Theorem}

	In the context of Theorem~\ref{thm:Hyperelliptic}, we have a hyperelliptic curve $X$ and a torsor $V_\ell$. One may identify $J = \Jac(X) = \Pic^0_X = \Pic^0_{V_\ell}$. When $X$ has a $K$-rational point, there is a embedding $X \hookrightarrow J$ sending the point to the identity element of $A(K)$. As we shall see, the class $\calB_{V_\ell}$ featuring in Theorem~\ref{thm:SpecializationOfBrauerClasses} pulls back via this embedding to give a class in $\Br(X)$ represented by the algebra $\calA_\ell$.
	\subsection{Notation}
		The following notation will be used throughout. 
		
		Let $K$ be a field of characteristic not equal to $2$, with a separable closure $\Kbar$ and absolute Galois group $G_K := \Gal(\Kbar/K)$. For a $G_K$-module $M$ (with the discrete topology) and an integer $i \ge 0$, $\HH^i(G_K,M)$ denotes the $i$th Galois cohomology group. For an \'etale $K$-algebra $L$ and elements $a,b \in L^\times$ define
		\[
			(a,b)_2 := \frac{L[i,j]}{\langle i^2 = a, j^2 = b, ij = -ij \rangle}\,.
		\]
		Kummer theory gives an isomorphism $\kappa:L^\times/L^{\times 2} \simeq \HH^1(G_K,\mu_2(\Lbar))$. The $L$-algebra $(a,b)_2$ represents the class of the cup product $\kappa(a) \cup \kappa(b)$ in $\HH^2(G_K,\mu_2(\Lbar)) = \Br(L)[2]$ (see Definition~\ref{def:BrauerGroup}). We denote this class by $[a,b]_2$.

	Suppose $Y$ is a smooth, projective and geometrically integral variety over $K$. We use $\Ybar$ to denote the base change of $Y$ to $\Kbar$. The function field of $Y$ is denoted $\kk(Y)$. Let $\Pic(Y)$ be the Picard group of $Y$. Then $\Pic(Y) = \Div(Y)/\Princ(Y)$, where $\Div(Y)$ (resp. $\Princ(Y)$) is the group of Weil divisors (resp. principal Weil divisors) of $Y$ defined over $K$. If $D \in \Div(Y)$, then $[D]$ denotes its class in $\Pic(Y)$. Let $\Pic_Y$ denote the (reduced) Picard scheme of $Y$, and let $\Pic^0_Y \subseteq \Pic_Y$ denote the connected component of the identity.  There is a bijective map $\Pic(\Ybar)^{G_K} \to \Pic_Y(K)$, but in general the map $\Pic(Y) \to \Pic(\Ybar)^{G_K}$ will not be surjective. Let $\Pic^0(Y)$ be the subgroup of $\Pic(Y)$ mapping into $\Pic^0_Y(K)$. Then $\NS(Y) := \Pic(Y)/\Pic^0(Y)$ is the {N\'eron-Severi group} of $Y$. If $\lambda \in \NS(\Ybar)^{G_K}$, let $\Pic^\lambda_Y$ denote the corresponding component of the Picard scheme and use $\Pic^\lambda(Y)$ and $\Div^\lambda(Y)$ to denote the subsets of $\Pic(Y)$ and $\Div(Y)$ mapping into $\Pic^\lambda_Y(K)$. We write $\Alb_Y$ for the Albanese scheme of $Y$ parameterizing equivalence classes of zero-cycles on $Y$ and, for $i \in \Z$, write $\Alb^i_Y$ for the degree $i$ component of $\Alb_Y$. Then $\Alb^0_Y$ is an abelian variety, and its dual abelian variety is $\Pic^0_Y$. When $Y$ is a curve, $\NS(\Ybar) = \Z$, $\Pic^i_Y = \Alb^i_Y$ for all $i \in \Z$ and $\Jac(Y) := \Pic^0_Y = \Alb^0_Y$ is called the Jacobian of $Y$.

		A \defi{torsor} under an abelian variety $A$ over $K$ is a variety $V$, together with an algebraic group action of $A$ on $V$ such that $A(\Kbar)$ acts simply transitively on $V(\Kbar)$. For any torsor $V$ under $A$, there is an  $\Kbar$-isomorphism of torsors $\Vbar \simeq \overline{A}$, determined up to translation by an element in $A(\Kbar)$. Since the action of a translation on $\Pic^0(\overline{A})$ is trivial, this determines an isomorphism of $G_K$-modules $\Pic^0(\Vbar) \simeq \Pic^0(\overline{A})$ which is independent of the choice of translation. For any $i \in \Z$, $\Alb^i_Y$ is a torsor under $\Alb^0_Y$. There is a canonical map $Y \to \Alb^1_Y$ induced by sending $y \in Y(\Kbar)$ to the class of the $0$-cycle $y$. This gives a canonical isomorphism of $G_K$-modules, $\Pic^0_Y(\Kbar) \simeq \Pic^0_{\Alb^1_Y}(\Kbar)$.

		 The isomorphism classes of torsors under $A$ are parametrized by the Weil-Ch\^atelet group, $\HH^1(K,A(\Kbar))$. We define the \defi{period} of a torsor to be its order in the Weil-Ch\^atelet group, and define the period of $Y$ to be the period of $\Alb^1_Y$.
 
\section{The relative Brauer group of a variety}
	
	\subsection{The Brauer group of a variety}
		\begin{Definition}\label{def:BrauerGroup}
			Let $Y$ be a smooth, projective and geometrically integral variety over a field $K$. The \defi{algebraic Brauer group} of $Y$ is 
			\[
				\Br(Y) :=\ker\left(\textup{div}_*:\HH^2(G_K,\kk(\Ybar)^\times) \to \HH^2(G_K,\Div(\Ybar)^\times)\right)
			\]
			where $\textup{div}_*$ is induced by the map sending a rational function on $\Ybar$ to its divisor.
		\end{Definition}
	
		When $Y = \Spec(K)$ we abbreviate $\Br(\Spec(K))$ to $\Br(K)$. In this case $\diw_*$ is trivial and so $\Br(K) = \HH^2(G_K,\Kbar^\times)$ is the usual Brauer group of the field $K$. In general, the inflation-restriction sequence identifies $\HH^2(G_K,\kk(\Ybar)^\times)$ with the relative Brauer group, $\Br(\kk(\Ybar)/\kk(Y))$. Thus, $\Br(Y)$ can (and tacitly will) be viewed as the subgroup of $\Br(\kk(Y))$ consisting of classes that are unramified at all divisors and split by extension of scalars to $\Kbar$.  Although we will not need it, let us also mention the \'etale-cohomological definition of the algebraic Brauer group as
		\[
			\Br(Y) := \ker\left( \HH^2_\textup{\'et}(Y,\G_m) \to \HH^2_\textup{\'et}(\Ybar,\G_m)\right)\,.
		\]
		It is shown in \cite[Appendix]{Lichtenbaum69} when $Y$ is a curve, and in \cite[Lemme 14]{C-TS} in general, that these definitions are equivalent.
	
	\subsection{The basic exact sequence}
		The Hochschild-Serre spectral sequence gives rise to a well known exact sequence,
		\begin{equation}\label{eq:frakaX}
			0 \to \Pic(Y) \to \Pic(\Ybar)^{G_K} \stackrel{\mathfrak{a}_Y}\to \Br(K) \to \Br(Y) \stackrel{h_Y}\to \HH^1(G_K,\Pic(\Ybar)) \to \HH^3(G_K,\Kbar^\times)\,.
		\end{equation}
		Following \cite[Section $2$]{Lichtenbaum69} (in the case that $Y$ is a curve), we will give a more direct derivation of~\eqref{eq:frakaX} using Galois cohomology. The exact sequences, 
		\begin{equation}
			\label{eq:PicSequence}
			1 \to \kk(\Ybar)^\times/\Kbar^\times \stackrel{\divv}\to \Div(\Ybar) \to \Pic(\Ybar) \to 0\,,
		\end{equation}
		and
		\begin{equation*}
			1 \to \Kbar^\times \to \kk(\Ybar)^\times \to \kk(\Ybar)^\times/\Kbar^\times \to 1\,,
		\end{equation*}
		induce a commutative diagram with exact rows and columns,
		\begin{equation}\label{eq:bigdiagram}
			\xymatrix{
				&&\Br(Y) \ar@{^{(}->}[d]
				&\HH^1(G_K,\Pic(\Ybar))\ar[d]^\rho\\
				\HH^1\left(G_K,\frac{\kk(\Ybar)^\times}{\Kbar^\times}\right)\ar[r]^\delta
				&\Br(K) \ar[r]
				&\HH^2(G_K,\kk(\Ybar)^\times) \ar[r]^\epsilon\ar[d]^{\divv_*}
				&\HH^1\left(G_K,\frac{\kk(\Ybar)^\times}{\Kbar^\times}\right)
				\ar[d]\ar[r]
				&\HH^3(G_K,\Kbar^\times)\\
				&&\HH^2(G_K,\Div(\Ybar))\ar@{=}[r]
				&\HH^2(G_K,\Div(\Ybar)).
			}
		\end{equation}
		The image of $\Br(K)$ is contained in the kernel of $\divv_*$ (a consequence of the fact that the divisor of a constant function is $0$). Also, $\rho$ is injective, since $\HH^1(G_K,\Div(\Ybar))=0$ (a consequence of Shapiro's Lemma). From the diagram we see that $\epsilon$ induces a unique map $h_Y$ making~\eqref{eq:frakaX} exact at all terms to the right of $\Br(K)$. We define $\mathfrak{a}_Y$ to be the composition of $\delta$ and the connecting homomorphism arising from~\eqref{eq:PicSequence}. The connecting homomorphism is surjective (for the same reason that $\rho$ is injective) and $\delta$ is injective (by Hilbert's Theorem 90). This shows that~\eqref{eq:frakaX} is exact at $\Br(K)$. The map $\Pic(Y) \to \Pic(\Ybar)^{G_K}$ is injective by Hilbert's Theorem 90 and its image is equal to  $\textup{image}(\Div(\Ybar)^{G_K} \to \Pic(\Ybar)^{G_K})$, which is equal to $\ker(\mathfrak{a}_X)$ by definition. Therefore~\eqref{eq:frakaX} is exact. 
		
		\begin{Lemma}\label{lem:frakaX2}
			Suppose $[D] \in \Pic(\Ybar)^{G_K}$ is represented by a \Edit{Removed ``ample'' as suggested by the referee}base point free divisor $D \in \Div(\Ybar)$. Then the complete linear system of effective divisors linearly equivalent to $D$ forms a Brauer-Severi variety $S$ over $K$, and $\mathfrak{a}_Y([D])$ is the class of $S$ in $\Br(K)$.
		\end{Lemma}
		
		\begin{proof}
			Since $[D]$ is fixed by Galois, for each $\sigma \in G_K$ there exists $f_\sigma \in \kk(\Ybar)^\times$ such that $\diw(f_\sigma) = \sigma(D) - (D)$. For any $\sigma,\tau \in G_K$, the divisor of $a_{\sigma,\tau} := \sigma(f_\tau) f_\sigma/f_{\sigma\tau}$ is $0$, which implies $a_{\sigma,\tau} \in \Kbar^\times$. The definition of $\mathfrak{a}_Y$ shows that $\mathfrak{a}_Y([D])$ is represented by the $2$-cocycle, $(\sigma,\tau) \mapsto a_{\sigma,\tau}$.
			
			The set $\calL(D) = \left\{ f \in \kk(\Ybar)^\times \;:\; \diw(f) + D \ge 0 \right\} \cup \{0\}$ is a $\Kbar$-vector space. For each $\sigma \in G_K$, the map $\phi_\sigma: f \mapsto \sigma(f)f_\sigma$ is an automorphism of $\calL(D)$. For $\sigma, \tau \in G_K$, these maps satisfy $a_{\sigma,\tau}\phi_{\sigma\tau} = \phi_\sigma \circ \phi_\tau$, and using this one easily verifies that the maps $\tilde{\phi}_\sigma:\PP(\calL(D)) \to \PP(\calL(D))$ induced by the $\phi_\sigma$ satisfy Weil's criteria for Galois descent (see \cite[Theorem 1]{Weil}). We conclude that there is a $K$-variety $S$ and a $\Kbar$-isomorphism $\psi:S_\Kbar \simeq \PP(\calL(D))$ such that $\tilde{\phi}_\sigma = \psi^\sigma\circ\psi^{-1}$, for every $\sigma \in G_K$. In particular, $S$ is a Brauer-Severi variety, and its class in $\HH^1(G_K,\Aut(\PP(\calL(D))) \simeq \HH^1(G_K,\PGL(\calL(D)))$ is represented by the $1$-cocycle $\sigma \mapsto \tilde{\phi}_\sigma$. The image of $S$ in $\Br(K) = \HH^2(G_K,\Kbar^\times)$ is then given by the $2$-cocycle $(\sigma,\tau) \mapsto \sigma(\phi_\tau)\circ\phi_\sigma\circ\phi_{\sigma\tau}^{-1} = a_{\sigma,\tau}$. As shown above, this also represents $\mathfrak{a}_Y([D])$.
		\end{proof}

	\subsection{Specialization of Brauer classes}
		There is a natural {\em evaluation pairing},
		\[
			\textup{ev} : Y(K) \times \Br(Y) \To \Br(K)\,.
		\]
		For $P \in Y(K)$, define $\Br(Y,P) := \ker\left(\textup{ev}(P,\,):\Br(Y) \to \Br(K)\right)$.
	
		\begin{Lemma}\label{lem:EvalSplits}
			For any $K$-rational point $P \in Y(K)$ the map $h_Y$ of~\eqref{eq:frakaX} restricts to an isomorphism $\Br(Y,P) \simeq \HH^1(G_K,\Pic(\Ybar))$.
		\end{Lemma}
		
		\begin{proof}
			The evaluation map gives a retraction of the exact sequence
			\[
				\Br(K) \to \Br(Y) \stackrel{h_Y}\to \HH^1(G_K,\Pic(\Ybar))\,.
			\]
		\end{proof}

	\subsection{Period and index subgroups}
		Following \cite{CipKrash} we make the following definitions
		\begin{Definition}\label{def:perind}
			Let $Y$ and $\mathfrak{a}_Y$ be as in~\eqref{eq:frakaX}.
			\begin{itemize}
				\item[(i)] $\Br^0(Y/K)$ is the image of $\mathfrak{a}_Y$ restricted to $\Pic^0(\Ybar)^{G_K}$.
				\item[(ii)] $\edit{\per}(Y)$ is the image of $\Pic(\Ybar)^{G_K}$ in $\NS(\Ybar)$.
				\item[(iii)] $\edit{\ind}(Y)$ is the image of $\Pic(Y)$ in $\NS(\Ybar)$.
			\end{itemize}	
		\end{Definition}

		\begin{Remark}\label{rem:PI}
			If $Y$ is a curve of period $p$ and index $i$, then $\NS(\Ybar) \simeq \Z$ and under this isomorphism $\per(Y) = p\Z$ and $\ind(Y) = i\Z$. The integer $p$ (resp. $i$) is the least positive degree of a $K$-rational divisor class (resp. $K$-rational divisor) on $Y$.
		\end{Remark}
	
		By definition of $\Br^0(Y/K)$ we have a commutative diagram with exact rows
		\begin{equation}
			\xymatrix{
				0 \ar[r]& \Pic^0(Y) \ar[r]\ar[d]& \Pic^0(\Ybar)^{G_K} \ar[r]^{\mathfrak{a}_Y}\ar[d]& \Br^0(Y/K) \ar[r]\ar[d]& 0\\
				0 \ar[r]& \Pic(Y) \ar[r]& \Pic(\Ybar)^{G_K} \ar[r]^{\mathfrak{a}_Y}& \Br(Y/K) \ar[r] & 0
			}	
		\end{equation}
		Applying the snake lemma immediately yields an exact sequence
		\begin{equation}\label{eq:perind}
			0 \to \Br^0(Y/K) \To \Br(Y/K) \To \frac{\per(Y)}{\ind(Y)} \to 0\,.
		\end{equation}

		If $\phi:Y \to Z$ is a morphism of smooth projective and geometrically integral varieties over $K$. Then the identity map on $\Br(K)$ induces an injection $\Br(Z/K) \hookrightarrow \Br(Y/K)$ and $\phi$ induces a morphism of short exact sequences,
		\begin{equation}\label{eq:SESmorphism}
			\xymatrix{
				0 \ar[r]& \Pic^0(Z) \ar[r]\ar[d]^{\phi^*}& \Pic^0(\overline{Z})^{G_K} \ar[r]^{\mathfrak{a}_Z}\ar[d]^{\phi^*}& \Br^0(Z/K) \ar[r]\ar@{^{(}->}[d]& 0\\
				0 \ar[r]& \Pic^0(Y) \ar[r]& \Pic^0(\Ybar)^{G_K} \ar[r]^{\mathfrak{a}_Y}& \Br^0(Y/K) \ar[r] & 0.\\
			}
		\end{equation}
		This has the following consequences.
		\begin{Lemma}\label{lem:IsoSES}
			Let $\phi:Y \to Z$ be as above.
			\begin{enumerate}
				\item If $\phi$ is the canonical map $Y \to \Alb^1_Y$, then~\eqref{eq:SESmorphism} is an isomorphism of exact sequences.
				\item If $D \in \Div(\overline{Z})$ represents a class in $\Pic^0(\overline{Z})^{G_K}$ and $\phi^*D$ is principal, then $D$ is linearly equivalent to a $K$-rational divisor.
			\end{enumerate}
		\end{Lemma}

		\begin{proof}
			Both statements follow by applying the snake lemma to~\eqref{eq:SESmorphism}. For the first statement, one also uses the fact that $\phi^*:\Pic^0(\Alb^1_\Ybar)^{G_K} \to \Pic^0(\Ybar)^{G_K}$ is an isomorphism.
		\end{proof}

\section{Relative Brauer groups via specialization}

	\subsection{Brauer classes on the Picard variety}%
		\label{sec:BrauerClassesViaSpecialization}
		Suppose $A$ is an abelian variety with dual $B = \Pic^0_A$. There is a canonical choice of rational point on $A$, namely the identity $0_A$. So by Lemma~\ref{lem:EvalSplits}, $h_A$ restricts to an isomorphism $\Br(A,0_A) \simeq \HH^1(G_K,\Pic(\overline{A}))$. Therefore, the inclusion $\Pic^0(\overline{A}) \subset \Pic(\overline{A})$ induces a map
		\begin{equation}\label{eq:WCtoBr}
			\HH^1(G_K,B(\Kbar)) = \HH^1(G_K,\Pic^0(\overline{A})) \to \HH^1(G_K,\Pic(\overline{A})) \simeq \Br(A,0_A) 
		\end{equation}
		
		\begin{Definition}
			\label{def:B_X}
			For a torsor $V$ under an abelian variety $B$ with dual abelian variety $A$, let $\calB_V \in \Br(A,0_A)$ be the image of the class of $V$ under the map in~\eqref{eq:WCtoBr}. For a smooth projective and geometrically integral variety $Y$ with Picard variety $A = \Pic^0_Y$, define $\calB_Y = \calB_{\Alb^1_Y} \in \Br(A,0_A)$ where $\Alb^1_Y$ is the torsor under $B = \Alb^0_Y$ parameterizing classes of $0$-cycles of degree $1$ on $Y$.
		\end{Definition}
	
		\begin{Remark}
		When $Y$ is a torsor under $B$ these two definitions of $\calB_Y$ coincide, since in this case $B = \Alb^0_Y$ and $Y = \Alb^1_Y$. 
		\end{Remark}
	
		\begin{proof}[Proof of Theorem~\ref{thm:SpecializationOfBrauerClasses}]
			Let $Y$ be a smooth projective and geometrically integral variety, let $\calB_Y \in \Br(A,0_A)$ as in Definition~\ref{def:B_X} and let $V = \Alb^1_Y$. Evaluation of $\calB_Y$ gives a homomorphism $A(K) \to \Br(K)$, which by \cite[Theorem 3.5]{CipKrash} coincides with $\mathfrak{a}_V:\Pic^0(\Vbar)^{G_K} = A(K) \to \Br(K)$. Hence evaluation of $\calB_Y$ induces a short exact sequence,
			\begin{equation}
				\label{eq:ForAlb1}
				0 \to \Pic^0(V) \To A(K) \stackrel{\calB_Y}\To \Br^0(V/K) \to 0.
			\end{equation}
			Using Lemma~\ref{lem:IsoSES} we may replace $V$ by $Y$ in~\eqref{eq:ForAlb1} without affecting the exactness. The theorem then follows from the exactness of~\eqref{eq:perind}.
		\end{proof}

	\subsection{Brauer classes on a curve}

		\begin{Definition}\label{def:BrauerClassOnTheCurve}
			Let $X$ be a smooth projective and geometrically integral curve. A torsor $V$ under the Jacobian of $X$ is \defi{Brauer-like} if its image in $\HH^1(G_K,\Pic(\Xbar))$ lies in the image of the map $h_X$ in~\eqref{eq:frakaX}. For a Brauer-like torsor $V$, let $\calA_V \in \Br(X)$ be any class such that $h_X(\calA_V)$ is the image of $V$ in $\HH^1(G_K,\Pic(\Xbar))$.
		\end{Definition}
	
		We note that $\calA_V$ is only defined up to the image of $\Br(K)$ in $\Br(X)$. The relationship between $\calA_V$ and the class $\calB_V$ of Definition~\ref{def:B_X} is given by the following.
	
		\begin{Lemma}
			Suppose $X$ has a rational point $P \in X(K)$ and let $V$ be a torsor under the Jacobian $J$ of $X$. Then
			\begin{enumerate}
				\item $V$ is Brauer-like,
				\item there is a unique choice for $\calA_V$ which lies in $\Br(X,P)$ and 
				\item $\calA_V = i^*(\calB_V)$, where $i^*:X \to J$ is the map induced by sending $Q \in X(\Kbar)$ to $[Q-P]$.
			\end{enumerate}
		\end{Lemma}
	
		\begin{proof}
			By Lemma~\ref{lem:EvalSplits} $h_X$ gives an isomorphism $\Br(X,P) \simeq \HH^1(G_K,\Pic(\Xbar))$. So it is clear that $V$ is Brauer-like and that the choice for $\calA_V$ in $\Br(X,P)$ is unique. The cokernel of the map $\HH^1(G_K,\Pic^0(\Xbar)) \to \HH^1(G_K,\Pic(\Xbar))$ induced by the inclusion is $\HH^1(G_K,\NS(\Xbar)) = \HH^1(G_K,\Z) = 0$. So the map $i^*$ induces a commutative diagram
			\[
				\xymatrix{
					\Br(X,P) \ar[rr]^{h_X}
					&&\HH^1\left(G_K,\Pic(\Xbar)\right) \ar@{=}[r] 
					&\HH^1\left(G_K,\Pic^0(\Xbar)\right) \\
					\Br(J,0_J) \ar[rr]^{h_J} \ar[u]^{\iota*} 
					&&\HH^1\left(G_K,\Pic(\overline{J})\right) 
					&\HH^1\left(G_K,\Pic^0(\overline{J})\right) \ar[l] \ar[u]^{\iota^*} 
				}
			\]
			Both $h_X$ and $h_J$ are isomorphisms (again by Lemma~\ref{lem:EvalSplits}). The vertical map on the right of the diagram is also an isomorphism, since $i^*:\Pic^0(\overline{J}) \to \Pic^0(\Xbar)$ is an isomorphism. The lemma then follows from a diagram chase.
		\end{proof}

	\subsection{Lichtenbaum's pairings}
	For a proper, smooth, geometrically integral curve $X$ over $K$, Lichtenbaum defined in \cite{Lichtenbaum69} three pairings and showed that they are compatible in the sense that there is a commutative diagram:
	\begin{equation}\label{eq:LichtenbaumPairings}
	\begin{array}[c]{cccccc}
		\langle\,,\,\rangle_1: &\Pic^0(\Xbar)^{G_K} & \times & \HH^1(G_K,\Pic^0(\Xbar)) & \rightarrow & \Br(K) \\
		&\rotatebox{90}{$\hookrightarrow$} && \rotatebox{90}{$\leftarrow$}&& \rotatebox{90}{$=$}\\
		\langle\,,\,\rangle_2: &\Pic^0(X) & \times & \HH^1(G_K,\Pic(\Xbar)) & \rightarrow & \Br(K)\\
		&\rotatebox{90}{$\hookleftarrow$} && \rotatebox{90}{$\rightarrow$}&& \rotatebox{90}{$=$}\\
		\langle\,,\,\rangle_3: &\Pic(X) & \times & \Br(X) & \rightarrow & \Br(K)
	\end{array}
	\end{equation}

	The pairing $\langle\,,\,\rangle_1$ was first defined by Tate; its definition may be found in \cite[Section 3]{CipKrash}. The pairing $\langle\,,\,\rangle_2$ is induced by $\langle\,,\,\rangle_1$; that it is well defined is verified in \cite[Section 4]{Lichtenbaum69}. The third pairing is induced by the evaluation map. Specifically, for $D \in \Pic(X)$ and $\calA \in \Br(X)$ one has $\langle D,\calA\rangle_3 = \bigotimes_P \calA(P)^{\otimes n_P}$, where $\sum_{P}n_PP$ is an integral linear sum of closed points of $X$ representing the divisor class $D$ (see \cite[Section 4]{Lichtenbaum69}). The vertical maps of the left column of~\eqref{eq:LichtenbaumPairings} are the obvious ones. The upper-middle map is induced by the inclusion $\Pic^0(\Xbar) \to \Pic(\Xbar)$, and the lower-middle map is $h_X$ of~\eqref{eq:frakaX}.

	The following \edit{lemma}\Edit{replaced "result" with "lemma" for clarity} relates the map $\mathfrak{a}_V$ with Tate's pairing. It was first proven by Lichtenbaum in the case of elliptic curves; for general abelian varieties see \cite[Theorem 3.5]{CipKrash}.
	\begin{Lemma}\label{lem:Tate=frakaX}
		For $V \in \HH^1(G_K,\Pic^0(\Xbar))$ and $P \in \Pic^0(\overline{V})^{G_K} = \Pic^0(\Xbar)^{G_K}$ one has $\langle P,V \rangle_1 = \mathfrak{a}_V(P)$.
	\end{Lemma}
	
	\begin{Proposition}\label{thm:Main2}
		Let $X$ be a curve and let $V$ be a Brauer-like torsor under the Jacobian of $X$. Evaluation of $\calA_V$ induces an exact sequence
		\[
			0 \to \left(\Pic^0(V)\cap\Pic^0(X)\right) \to \Pic^0(X) \stackrel{\calA_V}{\To} \Br^0(V/K) \To \frac{\Br^0(X/K)}{\mathfrak{a}_X(\Pic^0(V))} \to 0,
		\]
		where $\Pic^0(V)$ is identified with a subgroup of $\Pic^0(\Xbar)^{G_K}$ via the canonical identification of $\Pic^0(\Xbar)$ and $\Pic^0(\Vbar)$.
	\end{Proposition}

	The proposition is proved below. First we state a corollary.

	\begin{Corollary}
		\label{cor:Curves}
		If $X$ is a curve such that $\Br^0(X/K) = 0$ and $V$ is a Brauer-like torsor under its Jacobian, then evaluation of $\calA_V$ induces an exact sequence,
		\[
			J(K) \stackrel{\calA_V}{\To} \Br(V/K) \to \frac{\per(V)}{\ind(V)} \to 0\,,
		\]
		where $J = \Jac(X)$ is the Jacobian of $X$.
	\end{Corollary}


	\begin{proof}[Proof of Corollary~\ref{cor:Curves}]
	 	If $\Br^0(X/K) = 0$, then $\Pic^0(X) = \Pic^0(\Xbar)^{G_K} = J(K)$. By \edit{Proposition~\ref{thm:Main2}}\Edit{Reference corrected as suggested by referee}, evaluation of $\calA_V$ gives a surjective map $J(K) \to \Br^0(V/K)$. The exactness of~\eqref{eq:perind} then implies that the sequence in the corollary is exact.
	\end{proof}
	
	\begin{proof}[Proof of Proposition~\ref{thm:Main2}]
		Let $P \in \Pic^0(X) \subset \Pic^0(\Xbar)^{G_K}$. First note that $\langle P,\calA_V\rangle_3$ does not depend on the choice for $\calA_V$, since $P$ is a divisor class of degree $0$. By the commutativity in~\eqref{eq:LichtenbaumPairings} and Lemma~\ref{lem:Tate=frakaX} we have $\langle P, \calA_V \rangle_3 = \langle P,V\rangle_1 = \mathfrak{a}_V(P)$. Hence evaluation of $\calA_V$ induces a commutative and exact diagram
		\[
			\xymatrix{
				0 \ar[r] 
				& \Pic^0(X) \ar[r]\ar[d]^{\langle\,\,,\,\calA_V\rangle_3}
				& \Pic^0(\Xbar)^{G_K} \ar[r]^{\mathfrak{a}_X}\ar@{>>}[d]^{\mathfrak{a}_V}
				& \Br^0(X/K) \ar[r] \ar[d] & 0 \\
				0 \ar[r]
				& \Br^0(V/K) \ar@{=}[r]
				& \Br^0(V/K) \ar[r] & 0\,.			
			}
		\]
		Every element $a \in \Br^0(V/K)$ may be written as $a = \mathfrak{a}_V([D])$ for some $[D] \in \Pic^0(\Xbar)^{G_K}$, with $[D]$ determined uniquely modulo $\Pic^0(V)$. Therefore the map 
		\[
			r:\Br^0(V/K) \to \Br^0(X/K)/\mathfrak{a}_X(\Pic^0(V))
		\] 
		sending $a$ to $\mathfrak{a}_X([D])$ is well defined. Exactness in the statement in the proposition now follows by applying the snake lemma to the diagram above.
	\end{proof}

\section{Hyperelliptic curves}
\label{sec:HyperellipticCurves}

	Consider a hyperelliptic curve $X:y^2 = cf(x)$, where $f(x) \in K[x]$ is a square free monic polynomial over a field of characteristic not equal to $2$. The $x$-coordinate defines a degree $2$ map $\pi:X \to \PP^1_K$, embedding $K(x)$ into $\kk(X)$ as an index $2$ subfield. We set $L = K[\theta]/f(\theta)$\Edit{The usage of $L$ and $L_\circ$ has been clarified/corrected, as suggested by the referee}, and for $\ell \in L^\times$ set 
\[
	\calA_\ell := \left(\Cor_{L(x)/K(x)}(\ell,x-\theta)_2\right) \otimes_{K(x)}\kk(X)\,.
\]
	This is a central simple algebra over $\kk(X)$ which can be written explicitly as a tensor product of quaternion algebras (see~\cite[Proposition 2.4]{CVCurves}). In this section we show that for suitable torsors $V$ the class $\calA_V \in \Br(X)$ can be represented by an algebra of the form $\calA_\ell$.

	We say that $X$ is \defi{odd} if $\deg(f)$ is odd. Otherwise we say that $\edit{X}$ is \defi{even}. If $X$ is odd, we can perform a change of coordinates to arrange that $c = 1$. Let $\Omega \subseteq X$ be the set of ramification points of $\pi$, and let $\edit{L_\circ} = \Map_K(\Omega,\Kbar)$ denote the \'etale $K$-algebra corresponding to $\Omega$. When $X$ is even we may identify $L$ and $L_\circ$. When $X$ is odd, \edit{$L$} can be identified with the subalgebra \edit{of $L_\circ$} consisting of elements in $\Map_K(\Omega,\Kbar)$ that take the value $1$ at the ramification point above $\infty \in \PP^1(K)$. In the odd case this gives an isomorphism $\edit{L_\circ \cong L \times K}$.
	
	Set 
	\[
		\mathfrak{L} = \frac{L_\circ^\times}{K^\times L_\circ^{\times 2}}
	\]
	\edit{where $K^\times \subset L_\circ^\times = \Map_K(\Omega,\Kbar^\times)$ here denotes the multiplicative group of nonzero constant maps.}
	For $a \in K^\times$ and $\ell \in L^\times$, we use $\overline{a}$ and $\overline\ell$ to denote the corresponding classes in $K^\times/K^{\times2}$ and $\mathfrak{L}$, and set
	\[  
		\mathfrak{L}_a = \left\{ \overline{\ell} \in \mathfrak{L} \,:\, \Norm_{L/K}\left(\overline{\ell}\right)\in\langle \overline{a}\rangle \right\}\,, 
	\]
	where $\Norm_{L/K}$ denotes the map $\mathfrak{L} \to K^\times/K^{\times 2}$ induced by the norm on $\edit{L_\circ}$. Note that when $X$ is odd we have an isomorphism $\mathfrak{L} \simeq L^\times/L^{\times 2}$ under which $\Norm_{L/K}$ coincides with the map induced by the norm on $\edit{L}$. 

	\subsection{Odd hyperelliptic curves}

		Suppose that $X$ is odd (i.e., that $f(x)$ has odd degree). Schaefer \cite{Schaefer}, building on work of Cassels \cite{CasselsGenus2} defined an isomorphism of $G_K$-modules between $J[2]$ and $\ker\left(\Norm_{L/K}:\mu_2(\edit{\overline{L}}) \to \mu_2(\Kbar)\right)$. From this one obtains an isomorphism $e: \mathfrak{L}_1 \simeq \HH^1(G_K,J[2])$. The Kummer sequence associated to $J$ gives a surjective map $\HH^1(G_K,J[2]) \to \HH^1(G_K,J(\Kbar))[2] = \HH^1(G_K,\Pic^0(\Xbar))[2]$. Furthermore, the $K$-rational point $0_X \in X(K)$ lying above $\infty \in \PP^1_K(K)$ gives rise to isomorphisms,
		\[
			\Br(X,0_X) \stackrel{h_X}\simeq \HH^1(G_K,\Pic(\Xbar)) \simeq \HH^1(G_K,\Pic^0(\Xbar)).
		\]
		Hence there exists a homomorphism $\gamma$, uniquely determined by the requirement that the following diagram commute.
		\begin{equation}\label{eq:oddcase}
			\xymatrix{
				\HH^1(G_K,\Pic^0(\Xbar)[2]) \ar@{>>}[r]
				&\HH^1(G_K,\Pic^0(\Xbar))[2]\\
				\mathfrak{L}_1 \ar[r]^\gamma\ar[u]^e_{\rotatebox{90}{$\sim$}}
				&\Br(X,0_X)[2] \ar[u]^{h_X}_{\rotatebox{90}{$\sim$}}
			}
		\end{equation}
		Given $\ell \in L^\times$ of square norm, let $V_\ell$ be a torsor under $J$ representing the image of $\ell$ in the top right corner of the diagram. \edit{When $f(x)$ has degree $3$ (so that $X=J$ is an elliptic curve in Weierstrass form) the class of this torsor coincides with the torsor constructed in~\eqref{eq:V_ell} (see Theorem~\ref{thm:Appendix}).}
		
		   In~\cite{CVCurves} it is shown that $\gamma$ is induced by the map $\ell \mapsto \calA_\ell$. The commutativity of the diagram therefore implies that $\calA_\ell$ represents the class $\calA_{V_\ell}$ of Definition~\ref{def:BrauerClassOnTheCurve}. Applying Corollary \ref{cor:Curves} we obtain the following.

		\begin{Theorem}
			Suppose $X$ is an odd hyperelliptic curve and $V$ is a torsor of period $2$ under $J$. Then $V \simeq V_\ell$ for some $\ell \in L^\times$ of square norm. The algebra $\calA_\ell$ represents the class $\calA_V \in \Br(X,0_X)$ and is split if and only if $V(K)\ne\emptyset$. Moreover, evaluation of $\calA_\ell$ on $K$-rational divisors of degree $0$ induces an exact sequence.
			\[
				\frac{J(K)}{2J(K)} \stackrel{\calA_\ell}\To \Br(V/K) \To \frac{\per(V)}{\ind(V)} \To 0.
			\]	
		\end{Theorem}

	\subsection{Even hyperelliptic curves}\label{subsec:even}
		The case of even hyperelliptic curves is somewhat more involved. Suppose $\ell \in L^\times$ represents a class in $\mathfrak{L}_c$ and let $n \in \{ 0,1\}$ be such that $c^nN_{L/K}(\ell) \in K^{\times 2}$.\Edit{As noted by the referee there was a missing exponent of $2$.} Then $\ell$, together with a choice of square root for $c^nN_{L/K}(\ell)$, can be used to construct a torsor $V_\ell$ under $J$. For $n=0$ this follows from work of Poonen and Schaefer~\cite{PoonenSchaefer}. When $X$ has genus $2$ Flynn, Testa and Van Luijk~\cite{FTVL} give an explicit construction of $V_\ell$ as an intersection of $72$ quadrics in $\PP^{15}$. For $n=1$ this is developed in~\cite{CreutzANTSX}. See also \cite[Theorems 23 and 24]{BGW} for an alternative description of this torsor. When $X$ has genus $1$, \edit{the proof of \cite[Proposition 5.4]{CreutzANTSX} shows that $V_\ell$ is given by the construction in~\eqref{eq:V_ell2}. In this case the} $V_\ell$ are the torsors produced by a $4$-descent as described in~\cite{MSS4d}.

		The class of $V_\ell$ in $\HH^1(G_K,J(\Kbar))$ may depend up to inverse on the choice of square root for $c^nN_{L/K}(\ell)$, but the image of this class in $\HH^1(G_K,J(\Kbar))/\langle \Pic^1_X\rangle$ does not. These torsors also have the property that $2V_\ell = n\Pic^1_X$ in $\HH^1(G_K,J(\Kbar))$. This gives rise to a homomorphism 
		\begin{equation}\label{eq:construct}
			\mathfrak{L}_c \To \left(\frac{\HH^1(G_K,J(\Kbar))}{\langle \Pic^1_X \rangle}\right)[2] \simeq \HH^1(G_K,\Pic(\Xbar))[2]\,.
		\end{equation}
		It follows from \cite{CVCurves} that this homomorphism coincides with the composition
		\[
			\mathfrak{L}_c \stackrel{\gamma}\To \Br(X)/\Br_0(X) \stackrel{h_X}\To \HH^1(G_K,\Pic(\Xbar))\,,
		\]
		where $\gamma$ is induced by the map $\ell \mapsto \calA_\ell$ and $\Br_0(X)$ denotes the image of $\Br(K)$ in $\Br(X)$. This implies that $V_\ell$ is Brauer-like and that the algebra $\calA_\ell$ is a representative for $\calA_{V_\ell}$. Combining this with Corollary~\ref{cor:Curves} we obtain the following.

		\begin{Theorem}
			\label{thm:highergenus}
			Suppose $\Br^0(X/K) = 0$ and that the norm of $\ell \in L^\times$ lies in $c^nK^{\times 2}$. Let $V_\ell$ be a torsor under $J$ corresponding to a choice of square root for $c^nN_{L/K}(\ell)$, and let $\calA_\ell$ be the $\kk(X)$-algebra defined above. Then evaluation of $\calA_\ell$ on $K$-rational degree $0$ divisors induces an exact sequence
			\[
				\frac{J(K)}{2J(K)} \stackrel{\calA_\ell}{\To} \Br(V_\ell/K) \To \frac{\per(V_\ell)}{\ind(V_\ell)} \To 0.
			\]
		\end{Theorem}
		
		It is worth mentioning that the homomorphism in~\eqref{eq:construct} \edit{need not be} surjective.\Edit{This paragraph has been changed to answer a question of the referee.} \edit{In general there exist} Brauer-like torsors of period $2$ under $J$ which are not of the form $V_\ell$. \edit{For example, see \cite[Proposition 5.1 and Remark 5.4]{CVCurves}}. We also note that the hypothesis $\Br^0(X/K)=0$ is satisfied quite frequently, as suggested by the following lemma.

		\begin{Lemma}
			Suppose $X$ is a hyperelliptic curve of even genus defined over a local or global field. Then $\Br^0(X/K) = 0$.
		\end{Lemma}

		\begin{proof}
			The hypothesis implies that $\Pic^0(X) \to \Pic^0(\Xbar)^{G_K}$ is an isomorphism; see \cite[Propositions 3.3 and 3.4]{PoonenSchaefer}.
		\end{proof}
		
		\section{Genus one curves}
		\label{sec:Genus1}

			We now specialize to the case that $X$ is a genus $1$ hyperelliptic curve.

\subsection{Elliptic curves}\label{subsec:EC2}
	If $X$ is odd, then it is an elliptic curve in Weierstrass form. In this section we prove Theorem~\ref{thm:Elliptic} which allows us to compute the relative Brauer group of any torsor under $X = \Jac(X)$ of period $2$ (assuming we are able to compute the group $X(K)$ of rational points on $X$). Recall \edit{from Section~\ref{subsec:EC}}\Edit{Clarified at the request of the referee.} that for $\ell \in L^\times$ of square norm, the corresponding torsor under $X = \Jac(X)$ is the curve $V_\ell \subset \PP^3_K$ defined by
	\begin{equation}
			V_\ell : Q_\ell^{(1)}(u,v,w) + t^2 = Q_\ell^{(2)}(u,v,w) = 0\,,		
	\end{equation}
	where, for $0 \le i \le 2$, $Q_\ell^{(i)}(u,v,w) \in K[u,v,w]$ are the ternary quadratic forms uniquely determined by requiring that $\sum \theta^iQ_\ell^{(i)} = \ell(u + \theta v + \theta^2 w)^2.$ Let $C_\ell \subset \PP^2_K$ be the conic defined by $Q^{(2)}_\ell(u,v,w) = 0$.
	\Edit{Remark 5.1 has been removed since this has been superseded by the Appendix.}
	\begin{Lemma}
		\label{lem:conic}
		Let $P_0 = (u_0:v_0:w_0) \in C_\ell(\Kbar)$ be any point, and let 
		\[
			D = \left(u_0:v_0:w_0: \sqrt{-Q^{(1)}_\ell(u_0,v_0,w_0)}\right) + \left(u_0:v_0:w_0:-\sqrt{-Q^{(1)}_\ell(u_0,v_0,w_0)}\right) \in \Div(\Vbar_\ell)\,.
		\]
		Then
		\begin{enumerate}
			\item The class $[D]$ of $D$ in $\Pic(\Vbar_\ell)$ is $G_K$-invariant, and
			\item $\mathfrak{a}_{V_\ell}([D])$ is represented by the class of the conic $C_\ell$ in  $\Br(K)$
		\end{enumerate}
	\end{Lemma}
	
	\begin{proof}
		Let $\pi:V_\ell \to C_\ell$ be the degree $2$ mapped obtained by projecting away from the point $(0:0:0:1) \in \PP^3_K$. Then $D = \pi^*P_0$ and for $\sigma \in G_K$ we have $\sigma(D) = \pi^*(\sigma(P_0))$ which shows that $D$ and $\sigma(D)$ are linearly equivalent. This proves the first statement. The second follows from Lemma~\ref{lem:frakaX2}.
	\end{proof}

We are now in a position to prove Theorem~\ref{thm:Elliptic}. 

	\begin{proof}[Proof of Theorem~\ref{thm:Elliptic}]
		If $V_\ell(K) \ne \emptyset$ the statement holds trivially. Hence we may assume that $V_\ell$ has no $K$-rational point. This implies (using Riemann-Roch) that $V_\ell$ has no $K$-rational divisor class of degree $1$. It follows that $\per(X)$ is generated by the image of any $K$-rational divisor class of degree $2$. By Lemma~\ref{lem:conic}, the class of $C_\ell$ is equal to $\mathfrak{a}_{V_\ell}([D])$, where $[D]$ is a $K$-rational divisor class of degree $2$ on $V_\ell$. So the class of $C_\ell$ in $\Br(K)$ must generate $\Br(V_\ell/K)/\Br^0(V_\ell/K)$. Since, by Theorem~\ref{thm:Hyperelliptic}, evaluation of $\calA_\ell$ induces a surjective map $X(K) \to \Br^0(V_\ell/K)$, we have shown that $\Br(V_\ell/K)$ is generated by the set $\left\{ \calA_\ell(P) \;:\; P \in X(K) \right\} \cup \{ [C_\ell] \}\,.$
		
	To prove the second statement of the theorem, note that when $V_\ell(K) = \emptyset$ we have $\per(V_\ell) = \ind(V_\ell)$ if and only if $\Pic^2(V_\ell)\ne \emptyset$, which is the case if and only if there is some $M \in \Pic^2(\Vbar_\ell)^{G_K}$ such that $\mathfrak{a}_{V_\ell}(M) = 0$. Since $\Pic^2(\Vbar_\ell)^{G_K} = [D] + \Pic^0(\Vbar_\ell)^{G_K}$ and $\mathfrak{a}_{V_\ell}$ is a homomorphism, we conclude that $\per(V_\ell) = \ind(V_\ell)$ if and only if $[C_\ell]$ lies in $\left\{ \calA_\ell(P) \;:\; P \in X(K) \right\}$.
	\end{proof}
	
	\subsection{Curves of index $4$}
		\label{subsec:C4}
		Now suppose $X$ is even, that is $X:y^2 = cf(x)$ where $f(x)$ is a square free monic quartic polynomial. As described in Section~\ref{subsec:even} an element $\ell \in L^\times$ such that $cN_{L/K}(\ell) \in K^{\times 2}$ gives rise to a torsor $V_\ell$ under $J$ such that $2V_\ell = X$ in the Weil-Ch\^atelet group of the Jacobian. This may be given explicitly as an intersection of quadric surfaces in $\PP^3$. Namely, $V_\ell : Q_\ell^{(2)}(t,u,v,w) = Q_\ell^{(3)}(t,u,v,w) = 0$, where for $0\le i \le 3$, $Q_\ell^{(i)}(t,u,v,w)$ are the quadratic forms defined by 
		\[
			\sum_{0\le i \le 3} \theta^iQ_\ell^{(i)}(t,u,v,w) = \ell(t + \theta u + \theta^2 v + \theta^3w)^2.
		\]
		This defines a genus one curve of index dividing $4$. 

		Conversely, given a genus one curve $V$ of index $4$, the complete linear system corresponding to any $K$-rational divisor of degree $4$ defines an embedding of $V$ in $\PP^3$ as an intersection of quadric surfaces. If $M_1$ and $M_2$ denote the symmetric matrices corresponding to these quadrics, then $\edit{g}(x) := \det(M_1x+M_2)$ is a quartic polynomial\Edit{$f$ replaced by $g$ as suggested by the referee}, $V$ is a torsor under the Jacobian of the hyperelliptic curve $X : y^2 = g(x)$, and $2V = X$ in the Weil-Ch\^atelet group of the Jacobian (see~\cite{MSS4d}). Moreover, there exists some $\ell \in L^\times$ of norm a square times the leading coefficient of $g(x)$ such that $V \simeq V_\ell$ (an explicit formula for $\ell$ is given in~\cite[Section 5]{Fisher4d}).

\Edit{Lemma 5.3 removed since, as the referee noted, it is a trivial consequence of Lemma 3.5}
		\begin{proof}[Proof of Theorem~\ref{thm:index4}]
			Statements (1), (2) and (3) in the theorem follow from the discussion above. It remains to prove (4).

By assumption $V$ has period $4$, so $X$ has period $2$. By construction, the index of $X$ divides $2$, so it must in fact be $2$. Then both $V$ and $X$ have equal period and index, so $\Br(V/K) = \Br^0(V/K)$ and $\Br(X/K) = \Br^0(X/K)$.

			Proposition~\ref{thm:Main2} gives an exact sequence
			\begin{equation}
				\label{eq:inproof}
				0 \to \Pic^0(V)\cap\Pic^0(X) \to \Pic^0(X) \stackrel{\calA_\ell}{\To} \Br^0(V_{\ell}/K) \stackrel{r}\To \frac{\Br^0(X/K)}{\mathfrak{a}_X(\Pic^0(V))} \to 0,
			\end{equation}
			where the map $r$ is induced by sending $\mathfrak{a}_V([D]) \in \Br^0(V/K)$ to the class of $\mathfrak{a}_X([D]) \in \Br^0(X/K)$. \edit{Let $J$ denote the Jacobian of $V$ and $X$. Since the pairing $\langle\,,\,\rangle_1$ is bilinear, Lemma~\ref{lem:Tate=frakaX} implies that $\mathfrak{a}_X(P) = 2\mathfrak{a}_V(P)$, for any $P \in J(K) \cong \Pic^0(\Vbar)^{G_K} \cong \Pic^0(\Xbar)^{G_K}$. This shows that $\ker(\mathfrak{a}_V) \subset \ker(\mathfrak{a}_X)$. Hence we have $\Pic^0(V) = \ker(\mathfrak{a}_V)\cap J(K) \subset \ker(\mathfrak{a}_X)\cap J(K) = \Pic^0(X)$ and $\mathfrak{a}_X(\Pic^0(V)) = 2\mathfrak{a}_{V}(\Pic^0(V))=0$. The fact that $\mathfrak{a}_X(P) = 2\mathfrak{a}_V(P) = \mathfrak{a}_V(2P)$ also shows that $r$ is induced by multiplication by $2$. From this we see that} \eqref{eq:inproof}\Edit{This paragraph has been rewritten (1) to address the removal of Lemma 5.3 above, and (2) to give a more detailed explanation of the containment $\Pic^0(V) \subset \Pic^0(X)$ as requested by the referee.} can be rewritten as
			\begin{equation}\label{eq:target}
				0 \to \Pic^0(V) \to \Pic^0(X) \stackrel{\calA_\ell}\To \Br(V/K) \stackrel{2} \To \Br(X/K) \To 0\,.
			\end{equation}
		\end{proof}

\section{Examples}
	\label{sec:Examples}
	
		All computations in the examples below were performed using Magma (described in \cite{magma}). The code used to perform these computations may be found in the source of the arXiv distribution of this article.

		\subsection{A torsor of index $2$ under an elliptic curve}
			\label{subsec:ind2}
			
			We begin by illustrating how the results of this paper can be used to reproduce calculations of~\cite{HaileHan}. Namely, we show how to compute the relative Brauer group of a torsor of index $2$ under an elliptic curve (provided generators for the Mordell-Weil group are known). We will compute the relative Brauer group of
			\[
				V/\Q:y^2 = -2x^4 +6x^2+4x-14\,,
			\]
			considered in~\cite[Example 4]{HaileHan}.
			
			The Jacobian of $V$ is the elliptic curve $X/\Q : y^2 = f(x)$, where $f(x) =  x^3 -27\calI x-27\calJ$ with $\calI = 372$ and $\calJ = 12528$ the classical rational invariants associated to the quartic $g(x) = -2x^4+6x+4x-14$. Let $L = \Q[\theta]/f(\theta)$. We can find $\ell \in L^\times$ such that $V \simeq V_\ell$ using~\cite[Proposition 3.1]{Cremona} (see also~\cite[Section 2]{CremonaFisher}). The {\em cubic invariant} of $g(x)$ is $-\frac{8}{3}\varphi + 32$, where $\varphi$ is a root of the resolvent cubic $h(x) = x^3 -3\calI x+\calJ$ of $g(x)$. Noting that $\theta = -3\varphi$, it then follows from~\cite[Theorem 11]{CremonaFisher} that $V \simeq V_\ell$ for
			\[
				\ell = \frac{8}{9}\theta + 32 \equiv 2\theta + 72 \pmod{ L^{\times 2}}
			\]
			The corresponding algebra $\calA_\ell$ is obtained by extension of scalars of the $\Q(x)$-algebra,
			\[
				\calA_\ell' = \Cor_{L(x)/\Q(x)}( (2\theta+72,x-\theta)_2)\,.
			\]
			Let $r(x) = 2x + 72$ and $a = -2^53^6 = f(x) \bmod r(x)$. By~\cite[Proposition 2.4]{CVCurves},
			\[
				[\calA_\ell'] = [f(x),r(x)]_2 + [r(x),a]_2 + [1,2]_2 + [2,a]_2\,.
			\]
			Since $(1,2)_2$, $(2,-2)_2$ and $(2,a)_2$ are split and $f(x)$ is a square in $\kk(X)$, it follows that
			\[
				[\calA_\ell] = [2x+72,-2]_2 = [x+36,-2]_2\,.
			\]
			(This should be compared with the algebra $A_g$ defined in~\cite[Section 2]{HaileHan}.)
			
			The group $E(\Q)$ is generated by the points $P_1 =\left(-72, 108 \right)$ and $P_2 =  \left(450,9288 \right).$ Evaluating $\calA_\ell$ at these points gives
			\begin{align*}
				[\calA_\ell(P_1)] &= [-36,-2]_2 = [-1,-1]_2, \quad \text{and}\\
				[\calA_\ell(P_2)] &= \left[486,-2 \right]_2 = [6,-2]_2 = 0\,.
			\end{align*}
			So, by Theorem~\ref{thm:Elliptic}, $\Br(V/\Q) \simeq \Z/2\Z$ is generated by $[-1,-1]_2$.

			\begin{Remark}
				In \cite[Example 4]{HaileHan} the Jacobian is given as $E : y^2 = x^3 - 12x^2 - 76x - 32$. The map $(x,y) \mapsto (9x-36,27y)$ defines an isomorphism between $E$ and $X$ under which $P_1$ and $P_2$ are the images of the points $(-4,4)$ and $(54,344)$, respectively.
			\end{Remark}

	\subsection{Torsors of period $2$ under elliptic curves with full rational $2$-torsion}
		Let us consider an elliptic curve over $K$ with full rational $2$-torsion, say
		\[
			J : y^2 = x(x-a)(x-b)\,,\quad \text{ with } a,b \in K^\times
		\]
		Then every torsor of period $2$ under $J$ can be written in the form
		\[
			V_\ell : \left\{
				\begin{array}{c}
					au^2 - mv^2 + nt^2 = 0\\
					bu^2 - mv^2 + mnw^2 = 0
				\end{array}\right\} 
				\subset \PP^3_{(u:v:w:t)}\,,
		\]
		where $\ell = (m,n,mn) \in  K^\times\times K^\times\times K^\times = L^\times$. The algebra $\calA_\ell$ corresponding to $\ell$ is
		\[
			\calA_\ell = (x, m )_2 \otimes (x-a,n)_2 \otimes (x-b, mn)_2
		\]		
		The class of $\calA_\ell$ in $\Br(\kk(X))$ can be written in any of the following equivalent ways:
		\begin{align*}
			\left[\calA_\ell\right] &=  \left[ x, n\right]_2 + \left[ x-a,m\right]_2,\\
					&= \left[ x,mn\right]_2 + \left[ x-b, m \right]_2,\\
					&= \left[ x-a,mn\right]_2 + \left[ x-b,n\right]_2.
		\end{align*} 
		
		Evaluating this class at the nontrivial points in $J[2] \subset J(K)$ we obtain the following elements of $\Br^0(V_\ell/K)$:
		\begin{align*}
			\left[\calA_\ell(0,0)\right] &= \left[-a,mn\right]_2 + \left[ -b,n\right]_2\,,\\
			\left[\calA_\ell(a,0)\right] &= \left[ a, mn \right]_2 + \left[ a-b,m\right]_2\,,\\
			\left[\calA_\ell(b,0)\right] &= \left[ b,n\right]_2 + \left[ b-a,m\right]_2\,.
		\end{align*}
		If $J(K) = J[2]$, then these generate $\Br^0(V_\ell/K)$ and $\Br(V_\ell/K)/\Br^0(V_\ell/K)$ is generated by the class of the conic $C_\ell : bu^2 - mv^2 + mnw^2$, which is $[ C_\ell ]= [ bmn,-n ]_2$.

		For a concrete example we may take $K = \Q$, $a = 16$ and $b = -16$. Then $J(\Q) = J[2]$, and for $\ell = (-11,3,-33)$ we find
		\[
			\Br^0(V_\ell/\Q) = \{ \left[-1,-33\right]_2, \left[2,3\right]_2, \left[-2,-11\right]_2, \left[1,1\right]_2 \}
		\]
		and $[ C_\ell ] = [33,-1]_2 \notin \Br^0(V_\ell/\Q)$. In particular, we recover the result of Cassels \cite{CasselsV} that $\per(V_\ell) \ne \ind(V_\ell)$. 

		Over an algebraically closed field, $2$-torsion points on the Jacobian of a curve give rise to unramified double covers of the curve. In the presence of a nontrivial relative Brauer group this can fail, as shown in the following lemma.

		\begin{Proposition}
			Let $V_\ell/\Q$ be the curve corresponding to $\ell = (-11,3,-33)$ and let $E$ be its Jacobian. There does not exist an unramified double cover of $V_\ell$ defined over $\Q$ (despite the fact that $E(\Q)[2] \simeq \Z/2\Z\times \Z/2\Z$).
		\end{Proposition}

		\begin{proof}
			Suppose $\pi:Y \to V_\ell$ is an unramified double cover. The kernel of $\pi^*:J \to \Jac(Y)$ is generated by some $P \in J[2](\Q)$ with $P \ne 0$. Lemma~\ref{lem:IsoSES}(2) implies that $P$ is represented by a $\Q$-rational divisor on $V_\ell$. However, as shown above, none of the nontrivial elements of $J[2]$ is represented by a $\Q$-rational divisor.
		\end{proof}

		\subsection{A generic torsor of period $2$ and index $4$}
			The presence of rational $2$-torsion in the preceding example simplifies the computations, but it is by no means necessary. We now compute the relative Brauer group of a torsor of period $2$ and index $4$ in an example where there is no rational $2$-torsion, but the Mordell-Weil group of the Jacobian is nontrivial (which has the potential to make the relative Brauer group larger).			
						
			Consider the elliptic curve $X: y^2 = f(x) = x^3 - 1296x + 11664$ labeled 37a in Cremona's Database. Set $L = \Q[\theta]/f(\theta)$ and $\ell = \frac{-11\theta^2 + 322\theta - 1584}{2} \in L^\times$. One can check that $\ell$ has norm $2^4\times 3^6\times 569^2$ and, hence, that $\ell$ gives rise (as described in Section~\ref{subsec:EC}) to a torsor $V_\ell$ of period $2$ under $X$. 
			
			The Mordell-Weil group of $X$ is generated by the point $(0,108)$ of infinite order. 	Carrying out the computations as described in the example of Section~\ref{subsec:ind2} we obtain $[\calA_\ell(P)] = [-1,-1]_2$, which shows that $\Br^0(V_\ell/\Q) \simeq \Z/2\Z$. On the other hand, the conic $C_\ell$ is defined by the vanishing of
			\[
				Q^{(2)}_\ell = 11u^2 - 644uv + 31680uw + 15840v^2 - 1091232vw + 24284448w^2\,,
			\]
			whose class in $\Br(\Q)$ is $[ -3,569]_2$, which ramifies at $2$ and $569$. In particular it is not equal to $[-1,-1]_2$. Thus $\Br(V_\ell/\Q) \simeq \Z/2\Z \times \Z/2\Z$ is strictly larger than $\Br^0(V_\ell/\Q)$, which shows that $V_\ell$ has index $4$.
			
		\subsection{A torsor with period and index $4$}\label{subsec:period4}
		
			Consider the genus one hyperelliptic curve $X/\Q$ defined by $y^2=  f(x)$, where $f(x) = 5x^4 + 3x^2 + x + 3.$ Set $L = \Q[\theta]/f(x)$. The element $\ell = 10\theta^3 -\theta^2 + 8\theta - 1 \in L^\times$ has norm $N_{L/\Q}(\ell) = 3920 \in 5\Q^{\times 2}$. So, as described in Section~\ref{subsec:G1}, $\ell$ corresponds to a torsor $V_\ell$ of index dividing $4$ under the Jacobian $J$ of $X$, which is the elliptic curve $J : y^2 = x^3 -63x -113$ with $J(\Q)$ free of rank $2$ generated by $P_1 = (-3,7)$ and $P_2 = (-6,7)$.
			
			The curve $X$ has no points over $\Q_5$, and so $X$ has period $2$. Since $2V_\ell = X$ in the Weil-Ch\^atelet group, $V_\ell$ has period and index $4$. We will use Theorem~\ref{thm:index4} to compute the relative Brauer group of $V_\ell$ and determine which $\Q$-rational points on the Jacobian are represented by $\Q$-rational divisors on $V_\ell$. Specifically we will show:
			\begin{Example}
			 	$\Br(V_\ell/\Q)$ is generated by $[5,7]_2 \in \Br(\Q)$ and a point $P \in J(\Q)$ is represented by a $\Q$-rational divisor on $V_\ell$ if and only if its image in $J(\Q)/2J(\Q)$ lies in the subgroup generated by $P_1$.
			\end{Example}

			While $X$ has no $\Q_5$-points, it is locally solvable at all primes not equal to $5$ and over the reals. It follows that $\Br(X/\Q) = 0$. Hence the points $P_1$ and $P_2$ must be represented by $\Q$-rational divisors on $X$. The group of such divisors are generated by differences of traces of $K$-rational points on $X$ as $K$ ranges over quadratic extensions of $\Q$. A naive search over small quadratic fields finds the $\Q(\sqrt{3})$-points
			\begin{align*}
				Q_0 &= \left(0,\sqrt{3}\right), \quad
				Q_1 = \left(\frac{1+\sqrt{3}}{2} , -\frac{7+2\sqrt{3}}{2}\right), \text{ and} \\
				Q_2 &= \left(\frac{34+19\sqrt{3}}{2} , \frac{5011+ 2888\sqrt{3}}{4}\right)\,.
			\end{align*}
			Let $\sigma$ denote the generator of $\Gal(\Q(\sqrt{3})/\Q)$. Then 
			\begin{align*}
				D_1 &= (Q_1 + \sigma(Q_1)) - (Q_0 + \sigma(Q_0)), \text{ and}\\
				D_2 &= (Q_2 + \sigma(Q_2)) - (Q_0 + \sigma(Q_0))
			\end{align*}
			are $\Q$-rational divisors of degree $0$ that represent $P_1, P_2 \in J(\Q) = \Pic^0(\Xbar)^{\Gal(\Q)}$.
			
			By Theorem~\ref{thm:index4} we have an exact sequence
			\[
				0 \to \Pic^0(V_\ell) \to \Pic^0(X) \stackrel{\calA_\ell}\to \Br(V_\ell/\Q) \to 0\,,
			\]
			where $\calA_\ell = \Cor_{\kk(X_L)/\kk(X)}(x-\alpha,\ell)_2$. Thus $\Br(V_\ell/\Q)$ is generated by
			$\calA_\ell([D_i])$ for $i = 1,2$. Since the $x$-coordinates of $Q_0$ and $\sigma(Q_0)$ are the same, $(x-\alpha)([D_i])$ is obtained by evaluating the minimal polynomial over $\Q$ of the $x$-coordinate of $Q_i$ at $\theta$. Thus we obtain
			\begin{align*}
				(x-\alpha)([D_1]) &= \theta^2 - \theta - 1/2\,,\\
				(x-\alpha)([D_2]) &= \theta^2 - 34\theta + 73/4\,.
			\end{align*}
			Using Rosset-Tate reciprocity (as described in \cite[Corollary 7.4.10 and Remark 7.4.12]{GS-csa}) we find that $\calA_\ell([D_1]) \equiv \sum_{i=1}^4\calB_i$ and $\calA_\ell([D_2]) \equiv \sum_{i=1}^4\calC_i$, where
			\begin{align*}
				\calB_1 &= [-483,-129]_2\\
				\calB_2 &= [-3,-31]_2\\
				\calB_3 &= [-401730,-645]_2\\
				\calB_4 &= [-72345,-155]_2\\
				\calC_1 &= [14717515766140493089374, 54870144093800907126]_2\\
				\calC_2 &= [-89459567484222,-94665799875986]_2\\
				\calC_3 &= [1051000391530337870055,274350720469004535630]_2\\
				\calC_4 &= [-6948230469054914235,-473328999379930]_2\,.
			\end{align*}
			 Computing Hilbert symbols we find that $\calA_\ell([D_1])$ is split, while $\calA_\ell([D_2])$ is Brauer equivalent to $[5,7]_2$, which ramifies only at $5$ and $7$.
			
			\begin{Remark}
				The torsor $V_\ell$ has points over $\Q_p$ for all primes $p$ other than $5$ and $7$. So in fact, we knew a priori that the relative Brauer group would be contained in the subgroup generated by $[5,7]_2$. We take this as strong evidence of the correctness of our computations.
			\end{Remark}

		\subsection{Relative Brauer groups of higher dimensional torsors}
		
			As a final example we compute $\Br^0(V_\ell/\Q)$ for a torsor of period $2$ under the Jacobian of a hyperelliptic curve of genus $2$. Let $X/\Q$ be the hyperelliptic curve defined by
			\[
				X: y^2 = f(x) = x^6 + x^5 + x^4 + x^3 + x^2 + x + 1\,,
			\]
			and let $J$ be its Jacobian. Using Magma \cite{magma} we compute that $J(\Q) \simeq \Z^2$ generated by divisors $P_1 = D_1 - \pi^*\infty$ and $P_2 = D_2 - \pi^*\infty$, where $D_1$ and $D_2$ are divisors on $X$ whose $x$-coordinates are the roots of $t^2+t+1$ and $t^2 - t + 1$, and $\pi^*\infty$ is the pull-back of $\infty \in \PP^1$ under the map $\pi : X \to \PP^1$ sending a point on $X$ to its $x$-coordinate.
			
			For any $\ell \in L = \Q[\theta]/f(\theta)$ of square norm we can compute $\calA_\ell$, and as described in the previous example, $\calA_\ell(P_1) = \Cor_{L/\Q}(\theta^2+\theta+1,\ell)$ and $\calA_\ell(P_2) = \Cor_{L/\Q}(\theta^2-\theta+1,\ell)$. Moreover, the classes of these algebras in $\Br(\Q)$ generate $\Br^0(V_\ell/\Q)$, by Theorem~\ref{thm:highergenus}. 
			
			For a concrete example, take $\ell = 3\theta^4 + 2\theta^2 + 2\theta + 3$ which has norm $43^2$. Then $\calA_\ell(P_1)$ ramifies at $2$ and $43$, while $\calA_\ell(P_2)$ is split. In particular, $\Br^0(V_\ell/\Q) \simeq \Z/2\Z$ and a rational point on $J$ is represented by a $\Q$-rational divisor on $V_\ell$ if and only if its image in $J(\Q)/2J(\Q)$ lies in the subgroup generated by $P_2$.

\begin{appendices}
	\edit{
	\section{Appendix}
		Let $E:y^2 = f(x)$ be an elliptic curve defined over a field $K$ of characteristic not equal to $2$, where $f(x) \in K[x]$ is a monic polynomial of degree $3$ and set $L = K[x]/f(x)$. The Weil pairing gives an isomorphism $E[2] \simeq \Hom(E[2],\mu_2) = \ker\left(\Norm_{L/K}:\mu_2(\Lbar)\to\mu_2(\Kbar)\right)$, thus inducing an isomorphism $e:\HH^1(K,E[2]) \simeq \mathfrak{L}_1 := \ker\left(\Norm_{L/K}:L^\times/L^{\times 2} \to K^\times/K^{\times 2}\right)$. Composing $e$ with the connecting homomorphism $\delta$ yields Cassel's $x-\theta$ map (see \cite[Theorems 1.1 and 1.2]{Schaefer}). This gives the following commutative and exact diagram.
		
		\[
			\xymatrix{
				0 \ar[r] & \frac{E(K)}{2E(K)} \ar[drr]_{x-\theta} \ar[rr]^\delta && \HH^1(K,E[2]) \ar@{=}[d]^e \ar[r]& \HH^1(K,E)[2] \ar[r] & 0\\
				&&& \mathfrak{L}_1 \ar[ur]_v 
			}
		\]

	\begin{Theorem}\label{thm:Appendix}
		 The map $v$ in the diagram above sends the class of $\ell \in L^\times$ to the class of the torsor $V_\ell$ defined by \eqref{eq:V_ell}.
	\end{Theorem}
	
	\begin{proof}
		We will avail ourselves of the notation introduced at the beginning of Section~\ref{sec:HyperellipticCurves}. Denote the elements of $\Omega$ by $0_E, \omega_1,\omega_2,\omega_3$. Given $\ell \in L^\times \subset L_\circ^\times = \Map_K(\Omega,\Kbar^\times)$ such that $N_{L/K}(\ell) = a^2$, define $D_\ell \subset \PP^3\times E$ by declaring that $\left((u_0:u_1:u_2:u_3), (x,y)\right) \in D_\ell$ if and only if
		\[
			u_0^2(x-x(\omega_i)) = \ell_{\omega_i} u_i^2 \text{, for $i = 1,2,3$, and } y = a u_1 u_2 u_3\,.
		\]
		(Here we are using $\ell_{\omega_i}$ to denote the value of $\ell \in \Map(\Omega,\Kbar^\times)$ at $\omega_i \in \Omega$.) This condition is invariant under the action of $G_K$, so that $D_\ell$ is defined over $K$. Projection onto the first factor gives an isomorphism of $D_\ell$ and $V_\ell$. Since $V_\ell$ is a complete intersection, $D_\ell$ is geometrically connected. Projection onto the second factor defines a morphism $\pi_\ell:D_\ell \to E$ which one checks is \'etale of degree $4$. This implies that $D_\ell$ is a smooth curve of genus $1$.
		
		Any $\alpha \in \Lbar^\times$ gives an automorphism $(u_0:u_1:u_2:u_3) \mapsto (\alpha_{0_E}u_0:\alpha_{\omega_1}u_1:\alpha_{\omega_2}u_2: \alpha_{\omega_3}u_3)$ of $\overline{\PP}^3$. If $\alpha^2 = \ell$, then this induces an isomorphism of $E$-schemes $\overline{D}_\ell \simeq \overline{D}_1$. The projection $D$ of $D_1$ onto $\PP^3$ contains the $K$-rational point $0_D := (0:1:1:1)$ which lies in the fiber above $0_E$ under the composition $\pi: D \simeq D_1 \to E$. Thus $\pi:(D,0_D) \to (E,0_E)$ is a degree $4$ isogeny of elliptic curves. The points in the kernel of this isogeny all lie on the hyperplane $u_0 = 0$. Since $4.0_D$ is also a hyperplane section, the points in the kernel sum to $0_D$. This shows that the kernel of $\pi$ is not cyclic. Thus the only possibility is that $D \simeq E$ and $\pi$ is multiplication by $2$.
		
		We have shown that $\pi_\ell:D_\ell \to E$ is the twist of $[2]:E \to E$ by the cocycle $\xi_\ell : \sigma \mapsto \sigma(\alpha)/\alpha$, where $\alpha^2 = \ell$.  The cocycle $\xi_\ell$ takes values in $\ker\left(\Norm_{L/K}:\mu_2(\Lbar)\to\mu_2(\Kbar)\right) \simeq E[2]$ and the image of its class under $e$ is the class of $\ell$ in $\mathfrak{L}_1$. On the other hand the isomorphism $\overline{D}_\ell \simeq \overline{D}_1 \simeq \overline{E}$ endows $D_\ell$ with the structure of an $E$-torsor over $K$ and the image of $\xi_\ell$ in $\HH^1(K,E)[2]$ is represented by $D_\ell \simeq V_\ell$.
		\end{proof}
	}
\end{appendices}


\end{document}